\documentclass[11pt, a4paper]{amsart}
\usepackage{amssymb, array, amsmath, amscd, pdfpages, enumerate, amsthm, fixltx2e, setspace, hyperref, bbm,overpic}
\usepackage[justification=centering,margin=0.5cm]{caption}

\DeclareMathOperator*{\Ran}{Ran}
\DeclareMathOperator*{\Ker}{Ker}
\DeclareMathOperator*{\Fix}{Fix}
\DeclareMathOperator*{\dist}{dist}

\DeclareMathOperator*{\R}{Re}

\newcommand{\dd}{\mathrm{d}}

\newcommand{\ps}{\sigma_{p}}

\newcommand{\B}{\mathcal{B}}

\newcommand{\T}{\mathbb{T}}
\newcommand{\RR}{\mathbb{R}}
\newcommand{\CC}{\mathbb{C}}
\newcommand{\QQ}{\mathbb{Q}}
\newcommand{\NN}{\mathbb{N}}

\newcommand{\DD}{\mathbb{D}}

\newcommand{\EF}{(U(t,s))_{t\ge s\ge0}}

\newtheorem{thm}{Theorem}[section]
\newtheorem{prp}[thm]{Proposition}
\newtheorem{lem}[thm]{Lemma}
\newtheorem{cor}[thm]{Corollary}
\theoremstyle{definition}
\newtheorem{rem}[thm]{Remark}
\newtheorem{ex}[thm]{Example}

\numberwithin{equation}{section}

\begin{document}
\title{Asymptotics for periodic systems}
\author[L. Paunonen]{Lassi Paunonen}
\address[L. Paunonen]{Department of Mathematics, Tampere University of Technology, PO.\ Box 553, 33101 Tampere, Finland}
\email{lassi.paunonen@tut.fi}

\author[D. Seifert]{David Seifert}
\address[D. Seifert]{St John's College, St Giles, Oxford\;\;OX1 3JP, United Kingdom}
\email{david.seifert@sjc.ox.ac.uk}

\begin{abstract}
This paper investigates the asymptotic behaviour of solutions of periodic evolution equations. Starting with a general result concerning the quantified asymptotic behaviour of periodic evolution families we go on to consider a special class of dissipative systems arising naturally in applications.
For this class of systems we analyse in detail the spectral properties of the associated monodromy operator, showing in particular that it is a so-called Ritt operator under a natural `resonance' condition.  This allows us to deduce from our general result a precise description of the asymptotic behaviour of the corresponding solutions. 
In particular, we present conditions for rational rates of 
convergence to periodic solutions in the case where the convergence fails to be uniformly exponential. 
We illustrate our general results by applying them to concrete problems including the one-dimensional wave equation with periodic damping.
\end{abstract}

\thanks{This work was carried out while the first author was visiting Oxford from January to June 2017. The visit was hosted by Professor C.J.K.\ Batty. 
L.\ Paunonen is funded by the Academy of Finland grant number 298182.}

\maketitle

\section{Introduction}\label{sec:intro} 

The aim of this paper is to study stability properties of solutions to non-autonomous periodic evolution equations. An important motivating example is the one-dimensional \emph{damped wave equation}, 
\begin{equation}\label{eq:wave}
\left\{\begin{aligned}
z_{tt}(s,t)&=z_{ss}(s,t)-b(s,t)z_t(s,t),\quad & (s,t)\in\Omega_+,\\
z(0,t)&=z(1,t)=0, &t>0,\\
z(s,0)&=u(s), \quad z_t(s,0)=v(s),& s\in (0,1).
\end{aligned}\right.
\end{equation}
Here $\Omega_+=(0,1)\times(0,\infty)$, $b$ is a suitable non-negative function and the initial data satisfy $u\in H_0^1(0,1)$ and $v\in L^2(0,1)$. It is well known that if $b$ is not the zero function but independent of $t$, then the \emph{energy}
$$E(t)=\frac12\int_0^1 |z_s(s,t)|^2+|z_t(s,t)|^2\,\dd s,\quad t\ge0,$$
associated with any solution satisfies 
\begin{equation}\label{eq:exp_energy}
E(t)\le Me^{-\beta t}\left(\|u\|_{H_0^1}^2+\|v\|_{L^2}^2\right),\quad t\ge0,
\end{equation}
for some constants $M,\beta>0$ which are independent of the initial data; see for instance~\cite{CheFul91}. Similarly, it has recently been observed \cite{CasCin14, RLTT16} that for periodically time-dependent systems the energy of the solutions decays with an exponential rate provided the region in $(s,t)$-space where the damping coefficient $b$ is strictly positive satisfies a certain \emph{Geometric Control Condition (GCC)}.  A similar phenomenon occurs in the context of wave equations with autonomous damping on higher-dimensional spatial domains, where uniform exponential energy decay is in fact characterised by validity of the GCC; see \cite{BarLeb92,BurGer97,RauTay74}. 
For autonomous damped wave equations there is moreover a rich literature investigating the situation where the GCC is violated, showing in particular that it is possible even in this case to obtain rates of energy decay for solutions corresponding to particular initial data; see for instance \cite{BD08,Bur98,Leb96}.  
To date, however, nothing is known about such non-uniform rates of decay in the non-autonomous case. Our principal aim in the present work is to narrow this gap.

In fact, in the non-autonomous setting the energy of the solutions of~\eqref{eq:wave}  generally no longer decays at a uniform exponential rate, even in the presence of a significant amount of damping; see for instance~\cite{RLTT16}.  Indeed, as our examples in Section~\ref{sec:wave} demonstrate, there is no reason to expect energy decay at all if the period of the damping coincides with the period of the undamped wave equation, and instead in this \emph{resonant} case, which will be of particular interest in what follows,
one merely obtains convergence 
to periodic solutions, which have constant but possibly non-zero energy. One of our main objectives is to obtain statements about the \emph{rate} at which this convergence takes place, both when the GCC holds and when it is violated. 

To investigate this problem we begin by viewing the damped wave equation \eqref{eq:wave} as a non-autonomous abstract Cauchy problem of the form
\begin{equation}
  \label{eq:ACPintro}
  \left\{\begin{aligned}
    \dot{z}(t)&=A(t)z(t),\quad &t\ge0,\\
    z(0)&=x.
  \end{aligned}\right.
\end{equation}
Here $x=(u,v)^T\in H_0^1(0,1)\times L^2(0,1)$ is the initial data and the operators $A(t)$ are the form 
$$A(t)=A_0-B(t)B(t)^*,\quad t\ge0,$$ 
where $A_0$ is the wave operator corresponding to the undamped wave equation and the periodic operator-valued function $B$ captures the effect of the damping. 
In Section~2, we introduce a general framework for the study of rates of convergence for periodic non-autonomous systems of the form~\eqref{eq:ACPintro}.
Our approach is based on studying the associated  \textit{evolution family} $\EF$.
The main result of the section, Theorem~\ref{thm:gen}, characterises the quantified asymptotic behaviour of the solutions of~\eqref{eq:ACPintro} in terms of the properties of the so-called \emph{monodromy operator} $U(\tau,0)$, where $\tau>0$ is the period of the function $B$. This result may be viewed as a quantified version of several earlier results concerning the stability of periodic evolution families; see~\cite{Vu95} and also~\cite{BatChi02a, BatHut99, CasCin14, Da78,Ha83, RLTT16} and the references therein.

Then in Section~3 we introduce a class of dissipative systems which includes the damped wave equation. For this class of systems we obtain precise upper and lower bounds for the `energy' of solutions in terms of natural quantities associated with the family $\{A(t):t\ge0\}$. We moreover analyse the spectral properties of the associated monodromy operator,  showing among other things that  $U(\tau,0)$ is a so-called \emph{Ritt operator} under the natural
 resonance condition that the period $\tau$ of the damping coincides with the period of the group generated by $A_0$. Based on these results we  then provide, in the form of Theorem~\ref{thm:special}, a detailed description of the asymptotic behaviour of the corresponding solutions.
 This result shows in particular that there is a rich supply of initial data for which the solution converges (faster than) polynomially to a periodic solution even when uniform exponential convergence is ruled out.

Finally, in Sections~\ref{sec:transport} and \ref{sec:wave}, we apply our general theory to
specific  periodic partial differential equations in one space dimension, 
namely the transport equation and the damped wave equation~\eqref{eq:wave}. These examples demonstrate 
that 
in many natural cases involving 
substantial damping at any given time, the solution of the non-autonomous system 
may well converge to a non-zero periodic solution, 
which, as discussed above,  is in stark contrast to the situation for autonomous systems.
The examples also show how, depending on the precise nature of the damping function $b$, different initial values can lead to different rates of convergence.

The notation we use is more or less standard throughout. In particular, we write $X$ for a generic complex Hilbert space, or occasionally for a general Banach space.  We write $\B(X)$ for the space of bounded linear operators on $X$, and given $T\in\B(X)$ we write $\Ker(T)$ for the kernel and $\Ran(T)$ for the range of $T$. We let $\Fix T=\Ker (I-T)$. If $A$ is an unbounded operator on $X$ then we denote its domain by $D(A)$. Furthermore, we write $\sigma(T)$ for the spectrum and $\sigma_p(T)$ for the point spectrum of $T$. The spectral radius of an operator $T$ is denoted by $r(T)$, and for $\lambda\in\CC\setminus\sigma(T)$ we write $R(\lambda,T)$ for the resolvent operator $(\lambda-T)^{-1}$. We occasionally make use of standard asymptotic notation, such as `little o'.   Finally, we denote by $\T$ the unit circle $\{\lambda\in\CC:|\lambda|=1\}$ and by $\DD$ the open unit disc $\{\lambda\in\CC:|\lambda|<1\}$.

\section{Asymptotics for general periodic systems}\label{sec:general}

Let $X$ be a Hilbert space. An \emph{evolution family} $(U(t,s))_{t\ge s\ge0}$ is a family $\{U(t,s)\in\B(X):t\ge s\ge0\}$ of bounded linear operators on $X$ such that $U(t,t)=I$ for all $t\ge0$, $U(t,r)U(r,s)=U(t,s)$ for $t\ge r\ge s\ge0$, and the map $(t,s)\mapsto U(t,s)x$ is continuous on $\{(t,s):t\ge s\ge0\}$ for all $x\in X$.  We say that the evolution family $(U(t,s))_{t\ge s\ge0}$ is \emph{bounded} if $\sup_{t\ge s\ge0}\|U(t,s)\|<\infty$. Evolution families arise naturally in the context of non-autonomous Cauchy problems of the form
\begin{equation}\label{eq:naCP}
\left\{\begin{aligned}
\dot{z}(t)&=A(t)z(t),\quad &t\ge0,\\
z(0)&=x,
\end{aligned}\right.
\end{equation}
where $A(t)$, $t\ge0$, are closed and densely defined linear operators, and the initial condition $x\in X$ is given. Indeed, if the family $\{A(t):t\ge0\}$ is sufficiently well-behaved then there exists an evolution family $(U(t,s))_{t\ge s\ge0}$ associated with the problem \eqref{eq:naCP} with the property that the function $z\colon\RR_+\to X$ of \eqref{eq:naCP} defined by 
\begin{equation}\label{eq:orbit}
z(t)=U(t,0)x,\quad t\ge0,
\end{equation}
satisfies \eqref{eq:naCP} in an appropriate sense, at least for certain initial values $x\in X$. 
As has been explained in Section~\ref{sec:intro} we shall be interested only in a rather particular type of family $\{A(t):t\ge0\}$, to be introduced formally in Section~\ref{sec:dissip} below, for which the evolution family $(U(t,s))_{t\ge s\ge0}$ is related to the family $\{A(t):t\ge0\}$ through a certain variation of parameters formula and the function $z\colon\RR_+\to X$ defined in \eqref{eq:orbit} solves \eqref{eq:naCP} in a natural weak sense.  We point out, however, that in general the relationship between the family $\{A(t):t\ge0\}$ and the associated evolution family is a rather delicate matter; see for instance \cite[Section~VI.9]{EngNag00book}, \cite{LatRan98} and \cite[Chapter~5]{Paz83}. This is in contrast with the autonomous case where $A(t)=A$ for all $t\ge0$ and $A$ is the generator of a $C_0$-semigroup $(T(t))_{t\ge0}$. Here we may take $U(t,s)=T(t-s)$ for $t\ge s\ge0$, and the function $z(t)=T(t)x$, $t\ge0$, is then the \emph{mild solution} of \eqref{eq:naCP} in the usual sense, and it is a so-called \emph{classical solution} if and only if $x\in D(A)$; see \cite[Section~3.1]{ABHN11}. 

The main result in this section, Theorem~\ref{thm:gen} below, may be viewed as a theorem about the asymptotic behaviour of orbits of evolution families. However, motivated by the particular class of problems to be introduced in Section~\ref{sec:dissip}, we refer to the function $z\colon\RR_+\to X$ defined in \eqref{eq:orbit} as the \emph{solution} of \eqref{eq:naCP}, and consequently the evolution family $(U(t,s))_{t\ge s\ge0}$ is said to be associated with the non-autonomous Cauchy problem \eqref{eq:naCP}. We are particularly interested in evolution families $(U(t,s))_{t\ge s\ge0}$ which are $\tau$-periodic for some $\tau>0$ in the sense that $U(t+\tau,s+\tau)=U(t,s)$ for all $t\ge s\ge0$. This situation will arise in our concrete setting of Section~\ref{sec:dissip} if the family $\{A(t):t\ge0\}$  is $\tau$-periodic. In this case it is natural to consider the so-called \emph{monodromy operator} $T=U(\tau,0)$, and in particular the large-time asymptotic behaviour of the solution operators $U(t,0)$ as $t\to\infty$ is determined by the behaviour as $n\to\infty$ of the powers $T^n$ of the monodromy operator; see instance~\cite{Vu95}.  

We say that a function $z\colon\RR_+\to X$ is  \emph{asymptotically periodic} if there exists a periodic function $z_0\colon\RR_+\to X$ such that $\|z(t)-z_0(t)\|\to0$ as $t\to\infty$, and we say that the convergence is \emph{superpolynomially fast} if $\|z(t)-z_0(t)\|=o(t^{-\gamma})$ as $t\to\infty$ for all $\gamma>0$. We say that the system \eqref{eq:naCP} is \emph{asymptotically periodic} if the solution $z(t)$, $t\ge0$, is asymptotically periodic for all initial values $x\in X$, and we say that the system is \emph{stable} if $\|z(t)\|\to0$ as $t\to\infty$ for all initial values $x\in X$.  Recall that for any power-bounded operator $T\in\B(X)$, the operator $I-T$ is \emph{sectorial} of angle (of most) $\pi/2$, so that the fractional powers $(I-T)^\gamma$ are well-defined for all $\gamma\ge0$; see \cite{HaTo10} for details. Our first result is a quantified asymptotic result in the spirit of~\cite{Vu95}.

\begin{thm}\label{thm:gen}
Consider the non-autonomous Cauchy problem \eqref{eq:naCP} on a Hilbert space $X$, and suppose that the evolution family  $(U(t,s))_{t\ge s\ge0}$ associated with this problem  is bounded and $\tau$-periodic for some $\tau>0$.  
Let $T=U(\tau,0)$ be the monodromy operator, and suppose that $\sigma(T)\cap\T\subseteq\{1\}$ and that 
\begin{equation}\label{eq:res}
\|R(e^{i\theta},T)\|=O(|\theta|^{-\alpha}),\quad \theta\to0,
\end{equation}
for some $\alpha\ge1$. Then $X=\Fix T\oplus Z$, where $Z$ denotes the closure of $\Ran(I-T)$, and if we let $P$ denote the  projection onto $\Fix T$ along $Z$, then for any initial value $x\in X$ the solution $z\colon \RR_+\to X$ of \eqref{eq:naCP} satisfies
\begin{equation}\label{eq:asymp}
\|z(t)-z_0(t)\|\to0,\quad t\to\infty,
\end{equation}
where $z_0\colon\RR_+\to X$ is the $\tau$-periodic solution of \eqref{eq:naCP} with initial condition $z_0(0)=Px$. In particular, the system \eqref{eq:naCP} is asymptotically periodic, and it is stable if and only if $\Fix T=\{0\}$.

Moreover,   
\begin{equation}\label{eq:exp}
\|z(t)-z_0(t)\|\le Me^{-\beta t}\|x\|,\quad t\ge0, \,x\in X,
\end{equation}
for some  $M, \beta>0$ if and only if $\Ran(I-T)$ is closed. In any case, if $x\in X$ is such that $x-Px\in\Ran(I-T)^\gamma$ for some $\gamma>0$ then
\begin{equation}\label{eq:poly}
\|z(t)-z_0(t)\|=o(t^{-\gamma/\alpha}),\quad t\to\infty.
\end{equation}
Furthermore, there exists a dense subspace $X_0$ of $X$ such that for all $x\in X_0$ the convergence in \eqref{eq:asymp} is superpolynomially fast.
\end{thm}

\begin{proof}
Since the evolution family $(U(t,s))_{t\ge s\ge0}$ is assumed to be bounded, there exists $C>0$ such that $\|U(t,s)\|\le C$ for $t\ge s\ge 0$. In particular, $\sup_{n\ge0}\|T^n\|\le C$, so the monodromy operator $T$  is power-bounded. It follows from the mean ergodic theorem \cite[Chapter~2, Theorem~1.3]{Kre85} that $X=\Fix T\oplus Z$. Since $\sigma(T)\cap\T\subseteq\{1\}$ we have that $\|T^n(I-T)\|\to0$ as $n\to\infty$ by the Katznelson-Tzafriri theorem \cite[Theorem~1]{KT86}. Hence $\|T^nx\|\to0$ as $n\to\infty$ for all $x\in \Ran(I-T)$, and by power-boundedness of $T$ the statement extends to all $x\in Z$. In particular, we deduce that 
\begin{equation}\label{eq:proj}
\|T^nx-Px\|=\|T^n(I-P)x\|\to0,\quad n\to\infty,
\end{equation}
for all $x\in X$. Given $t\ge0$, if we let $n\ge0$ be the unique integer such that $t-n\tau\in[0,1)$, then by periodicity and contractivity of $(U(t,s))_{t\ge s\ge0}$ we have
\begin{equation}\label{eq:per}
\|z(t)-z_0(t)\|=\|U(t-n\tau,0)(T^nx-Px)\|\le C \|T^nx-Px\|,
\end{equation}
which implies \eqref{eq:asymp}. 

Now let $S$ denote the restriction of $T$ to the invariant subspace $Z$ and recall that $X=\Fix T\oplus Z$. Then $\sigma(S)\subseteq\sigma(T)$ and hence $\sigma(S)\subseteq\DD\cup\{1\}$. Moreover, the operator $I-S$ maps $Z$ bijectively onto $\Ran(I-T)$, so by the inverse mapping theorem we have $1\in\CC\setminus\sigma(S)$ if and only if $\Ran(I-T)$ is closed.  Thus if $\Ran(I-T)$ is closed, then $r(S)<1$ and we may take $r\in(r(S),1)$ and find a constant $K>0$ such that $\|S^n\|\le Kr^n$ for all $n\ge0$. It follows from \eqref{eq:per} that for $t\ge0$ and $n\ge0$ such that $t-n\tau\in[0,1)$ we have 
$$\|z(t)-z_0(t)\|\le CKr^n\|x\|,$$
 so \eqref{eq:exp} holds for $M=CKr^{-1/\tau}$ and $\beta =\frac{1}{\tau}\log\frac{1}{r}$. On the other hand, if \eqref{eq:exp} holds for some $M,\beta>0$, then for $x\in Z$ we have
 $$\|S^nx\|=\|z(n\tau)\|\le Me^{-\beta n\tau}\|x\|,\quad n\ge0,$$
 and in particular $\|S^n\|<1$ for sufficiently large $n\ge0$. Hence $r(S)<1$ and $\Ran(I-T)$ is closed.

For the last part, note that by \cite[Theorem~3.19 and Remark~3.12]{Sei16} condition \eqref{eq:res} implies that for $x\in\Ran(I-T)$ we have $\|T^nx\|=o(n^{-1/\alpha})$ as $n\to\infty$. By iterating this result and applying the moment inequality \cite[Theorem~II.5.34]{EngNag00book} to the sectorial operator $I-T$ it is now straightforward to obtain \eqref{eq:poly}. Finally, consider the spaces $X_k=\Fix T\oplus \Ran(I-T)^k$, $k\ge1$, and let $X_0=\bigcap_{k\ge1} X_k$. Since  $X_k$ is dense in $X$ for each $k\ge1$, it follows from a straightforward application of the Esterle-Mittag-Leffler theorem \cite[Theorem~2.1]{Est84} that $X_0$ is also dense in $X$. By construction the convergence in \eqref{eq:naCP} is superpolynomially fast  for each $x\in X_0$, so the proof is complete.
\end{proof}

\begin{rem}\label{rem:gen}
\begin{enumerate}[(a)]
\item We remark that the restriction in the statement of Theorem~\ref{thm:gen} that $\alpha\ge1$ is natural, since if $1\in\sigma(T)$ then the standard lower bound $\|R(\lambda,T)\|\ge(\dist(\lambda,\sigma(T)))^{-1}$, $\lambda\in\CC\setminus\sigma(A)$, implies that no smaller values of $\alpha$ can arise.
\item\label{it:slow}  In the case where $\Ran(I-T)$ is not closed we can in fact say more. Indeed, in this case $r(S)=1$ and it follows from \cite[Theorem~1]{Mue88} that for every sequence $(r_n)_{n\ge0}$ of non-negative terms converging to zero there exists $x\in Z$ such that $\|S^nx\|\ge r_n$ for all $n\ge0$. A simple argument as in the first part of the proof of \cite[Lemma~3.1.7]{vN96} now shows the convergence in \eqref{eq:asymp} is in fact arbitrarily slow in the sense that for any function $r\colon\RR_+\to[0,\infty)$ such that $r(t)\to0$ as $t\to\infty$ there exists $x\in X$ such that $\|z(t)-z_0(t)\|\ge r(t)$ for all $t\ge0$. So we have a dichotomy for the rate of decay: either it is uniformly exponentially fast, or it is arbitrarily slow for suitable initial values.
\item\label{it:proj} It follows from \eqref{eq:proj} that the projection $P$ onto $\Fix T$ along $Z$ satisfies $\|P\|\le \sup_{n\ge0}\|T^n\|$. In particular, if $T$ is a contraction then the projection $P$ is orthogonal.
\end{enumerate}
\end{rem}

It is straightforward for any $\alpha\ge1$  to construct examples of families $\{A(t):t\ge0\}$ of suitable multiplication operators to which Theorem~\ref{thm:gen} can be applied. In the next section we consider a special class of operators $A(t)$, $t\ge0$,  which are useful in applications and to which Theorem~\ref{thm:gen} can be applied with $\alpha=1$.

\section{A class of dissipative systems}\label{sec:dissip}

We now restrict our attention to the case where 
\begin{equation}\label{eq:sum}
A(t)=A_0-B(t)B(t)^*,\quad t\ge0,
\end{equation}
with $D(A(t))=D(A_0)$, $t\ge0$. Here $A_0$ is  assumed to be the infinitesimal generator of a unitary group $(T_0(t))_{t\in\RR}$ on $X$ and $B\in L_{\mathrm{loc}} ^2(\RR_+;\B(V,X))$ for some Hilbert space $V$. In particular, the operators $A(t)$, $t\ge0$, are dissipative. It follows from the Lumer-Phillips theorem and the results in  \cite[Chapter~5]{Paz83}, and in particular from \cite[Remark~5.3.2]{Paz83}, that there exists an evolution family $(U(t,s))_{t\ge s\ge0}$ of contractions associated with \eqref{eq:naCP} in the sense that the function $z\colon\RR_+\to X$ defined by $z(t)=U(t,0)x$, $t\ge0$, satisfies the variation of parameters formula
\begin{equation}\label{eq:vop}
z(t)=T_0(t)x-\int_0^tT_0(t-s)B(s)B(s)^*z(s)\,\dd s,\quad t\ge0
\end{equation}
 and hence may be viewed as a \emph{mild solution} of \eqref{eq:naCP}. As is easily verified, this mild solution can moreover be thought of as a weak solution of \eqref{eq:naCP} in the sense that for every $y\in D(A_0^*)$ the map $t\mapsto(z(t),y)$ is absolutely continuous on $\RR_+$ and 
$$\frac{\dd}{\dd t}(z(t),y)=(z(t),A(t)^*y)$$
for almost all $t\ge0$. We begin with a simple lemma which will be useful in studying the asymptotic behaviour of the solution of \eqref{eq:naCP}. 

\begin{lem}\label{lem:energy}
Let $A(t)$, $t\ge0$, be as in \eqref{eq:sum} and let $\tau>0$. Then 
$$\frac{\|x\|^2-\|U(\tau,0)x\|^2}{2}=\int_0^\tau\|B(t)^*U(t,0)x\|^2\,\dd t,\quad x\in X.$$
\end{lem}

\begin{proof}
If $B$ is constant on $(0,\tau)$ then the identity follows from the fundamental theorem of calculus for $x\in D(A_0)$, and by density it then holds for all $x\in X$. A similar argument applies when $B$ is a step function. Since $B\in L_{\mathrm{loc}} ^2(\RR_+;\B(V,X))$, a standard approximation argument yields the same identity in the general case. 
\end{proof}

\begin{rem}\label{rem:const_energy}
Let $\tau>0$ and $x\in X$.
  Then by Lemma~\ref{lem:energy} and the variation of parameters formula~\eqref{eq:vop} 
  we have that
$\|U(\tau,0)x\|=\|x\|$ if and only if $U(t,0)x=T_0(t)x$ for all $t\in[0,\tau]$.
\end{rem}

Let  $B^*\in L_{\mathrm{loc}} ^2(\RR_+;\B(X,V))$ be the function defined by $B^*(t)=B(t)^*$, $t\ge0$. Given a subset $Z$ of $X$ and a constant $\tau>0$ we say that the pair $(B^*, A)$ is \emph{approximately $Z$-observable on $(0,\tau)$} if for all $x\in Z$ the condition 
$$\int_0^\tau\|B(t)^*U(t,0)x\|^2\,\dd t=0$$
implies that $x=0$, and we say that $(B^*, A)$ is \emph{exactly $Z$-observable on $(0,\tau)$} if there exists a constant $\kappa>0$ such that
$$\int_0^\tau\|B(t)^*U(t,0)x\|^2\,\dd t\ge \kappa^2\|x\|^2$$
for all $x\in Z$. If $Z=X$ we simply call the pair $(B^*,A)$ approximately or exactly observable on $(0,\tau)$. For further discussion of observability and related concepts for non-autonomous systems see for instance \cite[Section~5]{Sch02}.

\begin{lem}\label{lem:obs}
Let $A(t)$, $t\ge0$, be as in \eqref{eq:sum} and let $\tau>0$. Then for all $x\in X$ we have
$$\frac{1}{c_\tau^{2}}\int_0^\tau\|B(t)^*T_0(t)x\|^2\,\dd t
\le \int_0^\tau\|B(t)^* U(t,0)x\|^2\,\dd t
\le \int_0^\tau\|B(t)^*T_0(t)x\|^2\,\dd t,$$
where $c_\tau=1+\|B\|_{L^2(0,\tau)}^2.$
In particular, given any subset $Z$ of $X$ the pair $(B^*,A)$ is approximately (respectively, exactly) $Z$-observable on $(0,\tau)$ if and only if  $(B^*,A_0)$ is approximately (respectively, exactly) $Z$-observable on $(0,\tau)$.
\end{lem}

\begin{proof}
 Consider the operators $\Phi_\tau,\Psi_\tau\in\B(X,L^2(0,\tau;V))$ given, for $x\in X$ and $t\in(0,\tau)$, by
$$(\Phi_\tau x)(t)= B(t)^*T_0(t)x\quad\mbox{and}
\quad (\Psi_\tau x)(t)= B(t)^*U(t,0)x.
$$
We show that there exists an isomorphism $Q_\tau\in\B(L^2(0,\tau;V))$ such that $\Phi_\tau=Q_\tau\circ\Psi_\tau$, and that moreover $\|Q_\tau\|\le c_\tau$ and 
$\|Q_\tau^{-1}\|\le 1$.
Indeed,  a straightforward calculation using  \eqref{eq:vop}  shows that
$$(\Psi_\tau x)(t) =(\Phi_\tau x)(t)- ((R_\tau\circ \Psi_\tau )x)(t)$$
for all $x\in X$ and almost all $t\in(0,\tau)$, where
$$(R_\tau y)(t)=\int_0^t B(t)^*T_0(t-s)B(s)y(s)\,\dd s$$
for $y\in L^2(0,\tau;V)$ and almost all $t\in(0,\tau)$. Thus $\Phi_\tau=Q_\tau\circ\Psi_\tau$, where $Q_\tau=I+R_\tau$, and a simple estimate gives  $Q_\tau\in\B( L^2(0,\tau;V))$ with $\|Q_\tau\|\le c_\tau$. We now show that $\R R_\tau\ge0$.  Let $y\in L^2(0,\tau; V)$. Then 
\begin{align*}
  \R(R_\tau y,y)
 =\int_0^\tau\int_0^t\R\big( T_0(-t)B(t)y(t), T_0(-s)B(s)y(s)\big)\,\dd s\,\dd t.
\end{align*}
Using Fubini's theorem to interchange the order of integration, we may rewrite the double integral to obtain
$$\R(R_\tau y,y)=\int_0^\tau\int_t^\tau\R\big( T_0(-t)B(t)y(t), T_0(-s)B(s)y(s)\big)\,\dd s\,\dd t.$$
Adding these two identities gives
$$\R(R_\tau y,y)=\frac12\left\|\int_0^\tau T_0(-t)B(t)y(t)\,\dd t\right\|^2\ge0,$$
as required. We now show that $Q_\tau$ is invertible. Indeed, $\Ran Q_\tau$ is dense because if $z\in L^2(0,\tau;V)$ is such that $(Q_\tau y,z)=0$ for all $y\in L^2(0,\tau;V)$, then in particular 
$$\|z\|^2\le \|z\|^2+\R(R_\tau z,z)=\R(Q_\tau z,z)=0,$$
so $z=0$. Moreover, 
$$\|y\|^2\le \R(Q_\tau y,y)\le \|Q_\tau y\|\|y\|$$
for all $y\in L^2(0,\tau;V)$, which shows that $\Ran Q_\tau$ is closed and that $Q_\tau$ is invertible with $\|Q_\tau^{-1}\|\le 1$. This completes the proof.
\end{proof}

Recall that an operator $T$ on a Banach space $X$ is said to be a \emph{Ritt operator} if $r(T)\le1$ and
$$\|R(\lambda,T)\|\le \frac{C}{|\lambda-1|},\quad |\lambda|>1,$$
for some constant $C>0$; see \cite{NaZe99}. It is shown in \cite{Lyu99, NaZe99} that $T$ is a Ritt operator if and only if $T$ is power-bounded and $\|T^n(I-T)\|=O(n^{-1})$ as $n\to\infty$. It is also known that a power-bounded operator is a Ritt operator if and only if $\sigma(T)\cap\T\subseteq\{1\}$ and \eqref{eq:res} holds with $\alpha=1$; see   \cite[Lemma~3.3]{CoLi16}. The next result provides the type of spectral information required in Theorem~\ref{thm:gen}; see \cite[Section V-1, Theorem 5.3]{BenDaP07book} for a related result on eigenvalues of monodromy operators. 

\begin{prp}\label{prp:spectrum}
Let $A(t)$, $t\ge0$, be as in \eqref{eq:sum} and suppose $T_0(\tau)=I$ for some $\tau>0$. Moreover, let $T=U(\tau,0)$.
Then 
$T$  is a Ritt operator
and 
\begin{equation}\label{eq:Fix}
\Fix T=\left\{x\in X:\int_0^\tau\|B(t)^*T_0(t)x\|^2\,\dd t=0\right\}.
\end{equation}
In particular, 
$\sigma(T)\cap\T\subseteq\{1\}$, and we have 
$1\not\in\ps(T)$ if and only if $(B^*,A_0)$ is approximately observable on $(0,\tau)$.
\end{prp}

\begin{proof}
 We show first that $\sigma(T)\cap\T\subseteq\{1\}$. Indeed, since $T$ is a contraction we know that $r(T)\le1$, and hence if $\lambda\in\sigma(T)\cap\T$ then $\lambda$ must be an approximate eigenvalue of $T$. In particular, we may find vectors $x_n\in X$, $n\ge1$, such that $\|x_n\|=1$ for all $n\ge1$ and $\|Tx_n-\lambda x_n\|\to0$ as $n\to\infty$. By Lemma~\ref{lem:energy} we have
$$\int_0^\tau\|B(t)^*U(t,0)x_n\|^2\,\dd t=\frac{\|x_n\|^2-\|Tx_n\|^2}{2}\to0,\quad n\to\infty.$$
Since $T_0(\tau)=I$, it follows from \eqref{eq:vop} that
$$\|Tx_n-x_n\|^2\le \|B\|_{L^2(0,\tau)}^2\int_0^\tau\|B(t)^*U(t,0)x_n\|^2\,\dd t\to0,\quad n\to\infty,
$$ 
and hence
$$|1-\lambda|\le\|Tx_n-x_n\|+\|Tx_n-\lambda x_n\|\to0,\quad n\to\infty,$$
so that $\lambda=1$, as required. 

Next we establish that $T$ is a Ritt operator. To this end let $x\in X$ with $\|x\|=1$ and let $\theta\in[-\pi,\pi]$. We first observe that
\begin{equation}\label{eq:prel_est}
|e^{i\theta}-1|\le \|e^{i\theta}x-Tx\|+|(Tx-x,x)|.
\end{equation}
 Using \eqref{eq:vop} and the fact that $T_0(\tau)=I$, so that in particular $T_0(\tau-t)^*=T_0(t)$ for $t\in[0,\tau]$,  we have
$$\begin{aligned}
|(Tx-x,x)|^2&=\left|\int_0^\tau(T_0(\tau-t)B(t)B(t)^*U(t,0)x,x)\,\dd t\right|^2\\
&\le\left(\int_0^\tau\|B(t)^*U(t,0)x\|^2\,\dd t\right)\left(\int_0^\tau\|B(t)^*T_0(t)x\|^2\,\dd t\right).
\end{aligned}$$
By Lemma~\ref{lem:energy} and the reverse triangle inequality we see that
\begin{equation}\label{eq:URitt}
\int_0^\tau\|B(t)^*U(t,0)x\|^2\,\dd t=\frac{\|x\|^2-\|Tx\|^2}{2}\le \|e^{i\theta}x-Tx\|,
\end{equation}
and hence by Lemma~\ref{lem:obs}
\begin{equation}\label{eq:TRitt}
\int_0^\tau\|B(t)^*T_0(t)x\|^2\,\dd t\le c_\tau^2\|e^{i\theta}x-Tx\|.
\end{equation}
Combining \eqref{eq:URitt} and \eqref{eq:TRitt} in the previous estimate we find that $|(Tx-x,x)|\le c_\tau\|e^{i\theta}x-Tx\|$, and hence \eqref{eq:prel_est} gives
$$|e^{i\theta}-1|\le (1+c_\tau)\|e^{i\theta}x-Tx\|.$$
 It  follows that $\|R(e^{i\theta},T)\|=O(|\theta|^{-1})$ as $\theta\to0$, so $T$ is a Ritt operator. 

In order to characterise the set $\Fix T$, note first that if $Tx=x$ then by Lemmas~\ref{lem:energy} and \ref{lem:obs} we have 
\begin{equation}\label{eq:no_damping}
\int_0^\tau\|B(t)^*T_0(t)x\|^2\,\dd t=0.
\end{equation}
Conversely, suppose that $x\in X$ is such that \eqref{eq:no_damping} holds.
Then Lemma~\ref{lem:obs} shows that 
$$\int_0^\tau\|B(t)^*U(t,0)x\|^2\,\dd t=0,$$
and it follows from \eqref{eq:vop} that
$$\|Tx-x\|^2\le \|B\|_{L^2(0,\tau)}^2\int_0^\tau\|B(t)^*U(t,0)x\|^2\,\dd t =0,$$
and hence $Tx=x$. In particular, we obtain \eqref{eq:Fix}, and hence $1\not\in\ps(T)$ if and only if $(B^*,A_0)$ is approximately observable on $(0,\tau)$.
\end{proof}

\begin{rem}
  Note that even without the assumption $T_0(\tau)=I$, approximate observability of 
$(B^*,A_0)$ on $(0,\tau)$ implies that $\ps( T )\cap\T=\emptyset$ for $T=U(\tau,0)$. 
Indeed, if $(B^*,A_0)$ is approximately observable on $(0,\tau)$ so is $(B^*,A)$ by Lemma~\ref{lem:obs}. Hence by Lemma~\ref{lem:energy} we have $\|Tx\|<\|x\|$ for all $x\in X\setminus\{0\}$, and in particular $\ps(T)\cap\T=\emptyset$. It follows from the Arendt-Batty-Lyubich-V\~{u} theorem~\cite{AreBat88} that  if the evolution family $\EF$ is $\tau$-periodic then the system~\eqref{eq:naCP} is stable whenever
$(B^*,A_0)$ is approximately observable on $(0,\tau)$ and the boundary spectrum $\sigma(T)\cap \T$ is countable.
\end{rem}

The next result establishes a connection between exact observability and the spectral radius of
certain restrictions of the monodromy operator.

\begin{prp}\label{prp:exp}
Let $A(t)$, $t\ge0$, be as in \eqref{eq:sum} and suppose that $B$ is $\tau$-periodic for some $\tau>0$. Moreover, let $T=U(\tau,0)$  and suppose that $Z$ is a closed $T$-invariant subspace of $X$. Then $r(T|_Z)<1$ if and only if  $(B^*,A_0)$ is exactly $Z$-observable on $(0,n\tau)$ for some $n\in\NN$. If $T_0(\tau)=I$ then $r(T|_Z)<1$ if and only if $(B^*,A_0)$ is exactly $Z$-observable on $(0,\tau)$.
\end{prp}

\begin{proof}
Let $S=T|_Z$. Then $r(S)<1$ if and only if there exists $n\in\NN$ such that $\|S^n\|=\|U(n\tau,0)|_Z\|<1$. Suppose that $n\in\NN$ is such that $\|S^n\|<1$ and let $x\in Z$. By Lemma~\ref{lem:energy} we have
$$\int_0^{n\tau}\|B(t)^*U(t,0)x\|^2\,\dd t=\frac{\|x\|^2-\|S^nx\|^2}{2}\ge\frac{1-\|S^n\|^2}{2}\|x\|^2,$$
and hence $(B^*,A)$ is exactly $Z$-observable on $(0,n\tau)$. By Lemma~\ref{lem:obs} the same is true of $(B^*,A_0)$. Now suppose conversely that $(B^*,A_0)$ is exactly $Z$-observable on $(0,n\tau)$ for some $n\in\NN$. By Lemma~\ref{lem:obs} the same is true of  $(B^*,A)$, and hence there exists a constant $\kappa>0$ such that
$$\int_0^{n\tau}\|B(t)^*U(t,0)x\|^2\,\dd t\ge \kappa^2\|x\|^2,\quad x\in Z.$$
Using Lemma~\ref{lem:energy} we deduce that
$$\|S^nx\|^2=\|x\|^2-2\int_0^{n\tau}\|B(t)^*U(t,0)x\|^2\,\dd t\le (1-2\kappa^2)\|x\|^2,\quad x\in Z,$$
and in particular $\|S^n\|<1$. Hence $r(S)<1$.  Suppose finally that $T_0(\tau)=I$ and that $r(S)<1$. If $n\in\NN$ is such that $\|S^n\|<1$, then by the first part we know that $(B^*,A_0)$ is exactly $Z$-observable on $(0,n\tau)$. Hence there exists a constant $\kappa>0$ such that 
$$\int_0^{n\tau}\|B(t)^*T_0(t)x\|^2\,\dd t\ge \kappa^2\|x\|^2,\quad x\in Z.$$
Since both $B$ and $T_0$ are $\tau$-periodic,
we have
$$\begin{aligned}
\int_0^{n\tau}\|B(t)^*T_0(t)x\|^2\,\dd t=n\int_0^{\tau}\|B(t)^*T_0(t)x\|^2\,\dd t.
\end{aligned}$$
Thus 
$$\int_0^{\tau}\|B(t)^*T_0(t)x\|^2\,\dd t\ge \frac{\kappa^2}{n}\|x\|^2,\quad x\in Z,$$
so $(B^*,A_0)$ is exactly $Z$-observable on $(0,\tau)$, as required.
\end{proof}

We now formulate a variant of Theorem~\ref{thm:gen} for 
operators $A(t)$, $t\ge0,$ which are of the form given in~\eqref{eq:sum}.

\begin{thm}\label{thm:special}
Suppose that the operators $A(t)$, $t\ge0,$ are as in \eqref{eq:sum} and that $B$ is $\tau$-periodic for some $\tau>0$. Suppose also that $T_0(\tau)=I$ and let $T=U(\tau,0)$ be the monodromy operator. Furthermore, let  $P$ denote the  orthogonal projection onto the closed subspace
$$Y=\left\{x\in X:\int_0^\tau\|B(t)^*T_0(t)x\|^2\,\dd t=0\right\}$$
of $X$ and let $Z=Y^\perp$. Then for any initial value $x\in X$ the solution $z\colon\RR_+\to X$ of \eqref{eq:naCP} satisfies
\begin{equation}\label{eq:asymp_spec}
\|z(t)-z_0(t)\|\to0,\quad t\to\infty,
\end{equation}
where $z_0\colon\RR_+\to X$ is the $\tau$-periodic solution of \eqref{eq:naCP} with initial condition $z_0(0)=Px$. In particular, the system \eqref{eq:naCP} is asymptotically periodic. The system is stable if and only if  $(B^*,A_0)$ is approximately observable on $(0,\tau)$.

Moreover,  
\begin{equation}\label{eq:exp_spec}
\|z(t)-z_0(t)\|\le Me^{-\beta t}\|x\|,\quad t\ge0,\, x\in X,
\end{equation}
for some  $M, \beta>0$ if and only if $(B^*,A_0)$ is exactly $Z$-observable on $(0,\tau)$. In any case, if $x\in X$ is such that $x-Px\in\Ran(I-T)^\gamma$ for some $\gamma>0$ then
\begin{equation}\label{eq:poly_spec}
\|z(t)-z_0(t)\|=o(t^{-\gamma}),\quad t\ge0.
\end{equation}
Furthermore, there exists a dense subspace $X_0$ of $X$ such that for all $x\in X_0$ the convergence in \eqref{eq:asymp_spec} is superpolynomially fast.
\end{thm}

\begin{proof}
The result follows immediately from Theorem~\ref{thm:gen}, Proposition~\ref{prp:spectrum} and Proposition~\ref{prp:exp}. Indeed, Proposition~\ref{prp:spectrum} shows that the monodromy operator $T$ is a Ritt operator, so that $\sigma(T)\cap\T\subseteq\{1\}$ and \eqref{eq:res} holds for $\alpha=1$, and moreover that $Y=\Fix T$, so that the system is stable if and only if $(B^*,A_0)$ is approximately observable on $(0,\tau)$. Note that by Remark~\ref{rem:gen}\eqref{it:proj}  the closure of $\Ran(I-T)$ coincides with the orthogonal complement $Z$ of $\Fix T$. From the proof of Theorem~\ref{thm:gen} it is clear that $\Ran(I-T)$ is closed if and only if the restriction $S=T|_Z$ of the monodromy operator $T$ to $Z$ satisfies $r(S)<1$. Hence by Proposition~\ref{prp:exp} the estimate in \eqref{eq:exp_spec} holds for some  $M,\beta>0$ if and only if $(B^*,A_0)$ is exactly $Z$-observable on $(0,\tau)$. The result now follows from Theorem~\ref{thm:gen}.
\end{proof}

\begin{cor}\label{cor:stable}
In the setting of Theorem~\textup{\ref{thm:special}}, the system \eqref{eq:naCP} is stable if and only if $(B^*,A_0)$ is approximately observable on $(0,\tau)$. Moreover,   
$$\|z(t)\|\le M e^{-\beta t}\|x\|,\quad t\ge0,\, x\in X,$$
for some  $M,\beta>0$ if and only if $(B^*,A_0)$ is exactly observable on $(0,\tau)$.
\end{cor}

\section{The transport equation}\label{sec:transport}

Let $\Omega=(0,1)\times\RR$ and $\Omega_+=(0,1)\times(0,\infty)$,  and consider the following initial-value problem for the transport equation subject to periodic boundary conditions,
\begin{equation}\label{eq:transport}
\left\{\begin{aligned}
z_t(s,t)&=z_s(s,t)-b(s,t)z(s,t),\quad & (s,t)\in\Omega_+,\\
z(0,t)&=z(1,t), &t>0,\\
z(s,0)&=x(s), & s\in(0,1),
\end{aligned}\right.
\end{equation}
where $x\in L^2(0,1)$ and $b\in L^\infty(\Omega)$ are given. We suppose that the damping term $b$ is 1-periodic in $t$ and that $b(s,t)\ge0$ for almost all $(s,t)\in\Omega$.  The problem can be cast in the form of \eqref{eq:naCP} with $A(t)$, $t\ge0$, as in \eqref{eq:sum}  by letting $X=L^2(0,1)$, $A_0x=x'$ for $x\in D(A_0)=\{x\in H^1(0,1):x(0)=x(1)\}$ and $B(t)x=b(\cdot,t)^{1/2}x$ for $t\ge0$ and  $x\in X$. Notice in particular that $A_0$ is the generator of the unitary group $(T_0(t))_{t\in\RR}$ given by $T_0(t)x=x(\cdot+t)$ for $x\in X$ and  $t\in\RR$. Here and in the remainder of this section any function on $(0,1)$ is identified with its 1-periodic extension to $\RR$. In particular, we have $T_0(1)=I$. The unique mild solution $z\colon\RR_+\to X$ of \eqref{eq:naCP} in the sense of \eqref{eq:vop} is given  by 
\begin{equation*}\label{eq:sol}
z(s,t)=x(s+t)\exp\left(-\int_0^tb(s+t-r,r)\,\dd r\right),\quad (s,t)\in\Omega_+.
\end{equation*}
Hence the monodromy operator $T=U(1,0)$ of the evolution family associated with the family $\{A(t):t\ge0\}$ is the multiplication operator corresponding to the function $m\in L^\infty(0,1)$ given by $m(s)=\exp(-a(s))$, where
$$a(s)=\int_0^1b(s-r,r)\,\dd r,\quad s\in(0,1).$$
Define, modulo null sets,  $I_a=\{s\in(0,1):a(s)>0\}$ and $J_a=\{s\in(0,1):a(s)=0\}$. Then $\Fix T=L^2(J_a)$ and the orthogonal projection $P$ onto $\Fix T$ is given simply by $Px=\mathbbm{1}_{J_a} x$, $x\in X$. Similarly, the orthogonal complement $Z=(\Fix T)^\perp$ of $\Fix T$ is given by $Z=L^2(I_a)$. In particular, $(B^*,A_0)$ is approximately observable on $(0,1)$ if and only if $J_a$ is a null set, and $(B^*,A_0)$ is exactly $Z$-observable on $(0,1)$ if and only if $a(s)\ge c$ for almost all $s\in I_a$ and some $c>0$ . For $\gamma>0$ we have 
$$\Ran (I-T)^\gamma=\{x\in Z:(1-m)^{-\gamma}x\in X\}=\{x\in Z:a^{-\gamma}x\in X\}.$$ 
These observations lead to the following special case of Theorem~\ref{thm:special}.

\begin{thm}\label{thm:transp}
Let $X=L^2(0,1)$ and let $I_a,J_a\subseteq(0,1)$ be as above. For any initial value $x\in X$ the solution $z\colon\RR_+\to X$ of the problem \eqref{eq:naCP} corresponding to \eqref{eq:transport} satisfies
\begin{equation}\label{eq:asymp_transp}
\|z(t)-(\mathbbm{1}_{J_a} x)(\cdot+t)\|\to0,\quad t\to\infty.
\end{equation}
In particular, the system \eqref{eq:naCP} is asymptotically periodic and it is stable if and only if  $J_a$ is a null set. 

Moreover, 
\begin{equation}\label{eq:exp_transp}
\|z(t)-(\mathbbm{1}_{J_a} x)(\cdot+t)\|\le Me^{-\beta t}\|x\|,\quad t\ge0,\, x\in X,
\end{equation}
for some  $M, \beta>0$ if and only if  $a(s)\ge c$ for almost all $s\in I_a$ and some $c>0$. In any case, if $x\in X$ is such that $a^{-\gamma}\mathbbm{1}_{I_a}x\in X$ for some $\gamma>0$ then
\begin{equation*}\label{eq:poly_transp}
\|z(t)-(\mathbbm{1}_{J_a} x)(\cdot+t)\|=o(t^{-\gamma}),\quad t\to\infty.
\end{equation*}
Furthermore, if $x$ lies in the dense subspace of functions satisfying $a^{-k}\mathbbm{1}_{I_a}x\in X$ for all  $k\ge1$ then the convergence in \eqref{eq:asymp_transp} is superpolynomially fast.
\end{thm}

We illustrate Theorem~\ref{thm:transp} in the case where $b=\mathbbm{1}_\omega$ is an indicator function of some measurable subset $\omega$ of $\Omega$. To ensure periodicity of our system we assume that  $\omega$ is translation invariant in the $t$-direction, so that $\omega+\{(0,1)\}=\omega.$ Thus $b$ is completely described by the set $\omega_0=\omega\cap(0,1)^2$. 

\begin{figure}[ht]
  \begin{minipage}[c]{.49\linewidth}
    \begin{center}
      \includegraphics[width=.81\linewidth]{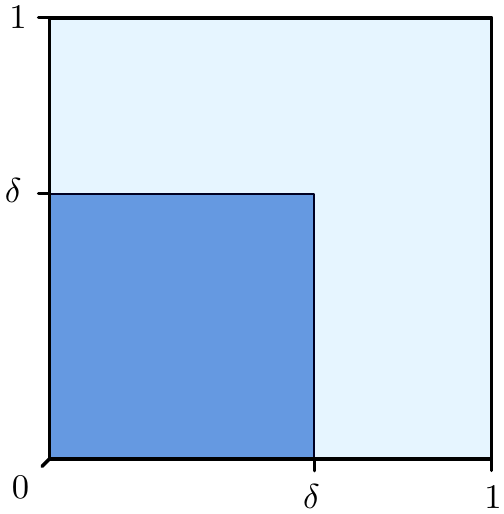}
    \end{center}
  \end{minipage}
  \hfill
  \begin{minipage}[c]{.49\linewidth}
    \begin{center}
      \includegraphics[width=.8\linewidth]{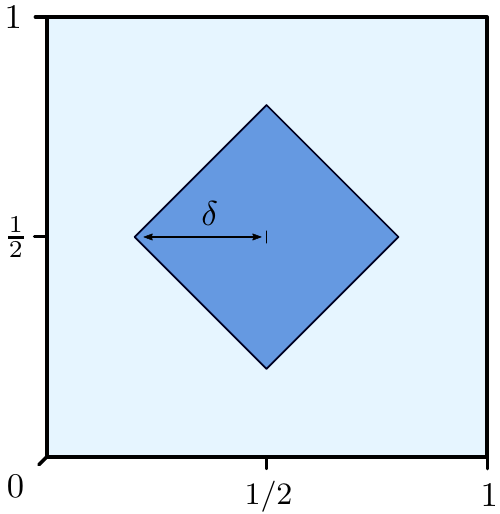}
    \end{center}
  \end{minipage}

    \caption{The damping regions $\omega_0$ for Examples~\ref{ex:TranspSquare} and~\ref{ex:TranspDiamond}\\ with $s$ on the horizontal axis and $t$ on the vertical axis. }
      \label{fig:Transport}
\end{figure}

\begin{ex}
  \label{ex:TranspSquare}
Suppose that 
$$\omega_0=\big\{(s,t)\in(0,1)^2:|s-1/2|+|t-1/2|<\delta\big\}$$
for some  $\delta\in[0,1/2]$. 
Then $a=(\delta/2)\mathbbm{1}_{(1-2\delta,1)}$, so $I_a=(1-2\delta,1)$ and $J_a=(0,1-2\delta)$. Thus the system \eqref{eq:naCP} corresponding to \eqref{eq:transport} is stable if and only if $\delta=1/2$, and in any case \eqref{eq:exp_transp} holds for some $M,\beta>0$.
\end{ex}

\begin{ex}
  \label{ex:TranspDiamond}
Suppose that 
$$\omega_0=\{(s,t)\in(0,1)^2:0<s,t<\delta\}$$
for some  $\delta\in[0,1]$. 
For $\delta\in[0,1/2)$ we have
 $$a(s)=\left\{\begin{aligned}
&s,\qquad& 0<s<\delta,\\
&2\delta-s,& \delta< s<2\delta,\\
&0,& 2\delta<s<1.
\end{aligned}\right.
$$
and for $\delta\in[1/2,1]$ we have
$$a(s)=\left\{\begin{aligned}
&2\delta-1,\quad& 0<s< 2\delta-1,\\
&s,\qquad& 2\delta-1<s<\delta,\\
&2\delta-s,& \delta< s<1.
\end{aligned}\right.$$
Thus for $\delta\in[0,1/2)$ we have $I_a=(0,2\delta)$ and $J_a=(2\delta,1)$, while for  $\delta\in[1/2,1]$ we have $I_a=(0,1)$ and $J_a=\emptyset$, so the system \eqref{eq:naCP} corresponding to \eqref{eq:transport} is stable if and only if $\delta\in[1/2,1]$. When $\delta\in(1/2,1]$ the solution $z\colon\RR_+\to X$ satisfies $\|z(t)\|\le Me^{-\beta t}$, $t\ge0$, for some  $M,\beta>0$, whereas for $\delta=1/2$ no such constants exist. In this case, however, we have $\|z(t)\|=o(t^{-\gamma})$ as $t\to\infty$ for $\gamma>0$ provided 
\begin{equation}\label{eq:poly_transp_ex}
\int_0^1\frac{|x(s)|^2}{(\min\{s,1-s\})^{2\gamma}}\,\dd s<\infty,
\end{equation}
and the convergence is superpolynomially fast if \eqref{eq:poly_transp_ex} holds for all $\gamma\in\NN$. This is the case in particular  
if there exists $\varepsilon>0$ such that $x(s)=0$ for almost all $s\in(0,\varepsilon)\cup(1-\varepsilon,1)$.
For $\delta\in[0,1/2)$ the system fails to be stable but it is still asymptotically periodic, and in fact $\|z(t)-(\mathbbm{1}_{J_a}x)(\cdot+t)\|\to0$ as $t\to\infty$. The convergence is not uniformly exponentially fast, but for $\gamma>0$ and $x\in X$ such that
\begin{equation}\label{eq:poly_transp_ex2}
\int_0^{2\delta}\frac{|x(s)|^2}{(\min\{s,2\delta-s\})^{2\gamma}}\,\dd s<\infty,
\end{equation}
we have $\|z(t)-(\mathbbm{1}_{J_a}x)(\cdot+t)\|=o(t^{-\gamma})$ as $t\to\infty$, and the convergence is superpolynomially fast if \eqref{eq:poly_transp_ex2} holds for all $\gamma\in\NN$, as is the case in particular if there exists $\varepsilon>0$ such that $x(s)=0$ for almost all $s\in(0,\varepsilon)\cup(2\delta-\varepsilon,2\delta)$.
\end{ex}

\section{The time-dependent damped wave equation}\label{sec:wave}

We return finally to the time-dependent damped wave equation introduced in Section~\ref{sec:intro}. Let $\Omega=(0,1)\times\RR$ and $\Omega_+=(0,1)\times(0,\infty)$, and assume that $b\in L^\infty(\Omega)$ with $b(s,t)\ge0$ for almost all $(s,t)\in\Omega$. Then the problem can be written in the form of \eqref{eq:naCP} with operators $A(t)$, $t\ge0$, as in \eqref{eq:sum} by choosing $X=H_0^1(0,1)\times L^2(0,1)$, $A_0x=(v,u'')^T$ for $x=(u,v)^T\in D(A_0)=(H^2(0,1)\cap H_0^1(0,1))\times H_0^1(0,1)$ and $B(t)x=(0,b(\cdot,t)^{1/2}v)^T$ for $x=(u,v)^T\in X$ and $t\ge0$. Note in particular that the unitary group $(T_0(t))_{t\in\RR}$ generated by $A_0$ satisfies $T_0(2)=I$ since solutions of the undamped wave equation on $(0,1)$ are 2-periodic. Indeed,  the undamped wave equation with initial data $x=(u,v)^T\in X$ can be solved explicitly using d'Alembert's formula, which in this case gives
\begin{equation}\label{eq:d'Alem}
z(s,t)=\frac{\tilde{u}(s+t)+\tilde{u}(s-t)}{2}+\frac12\int_{s-t}^{s+t}\tilde{v}(r)\,\dd r,\quad (s,t)\in\Omega,
\end{equation}
where $\tilde{u}$ and $\tilde{v}$ are the odd 2-periodic extensions to $\RR$ of $u$ and $v$, respectively. Furthermore, the energy of the solution $z\colon \RR_+\to X$ of \eqref{eq:naCP} satisfies 
$$E(t)=\frac12\|z(t)\|^2,\quad t\ge0.$$

Consider the special case where $b=\mathbbm{1}_\omega$ and suppose that $\omega$ is (up to a null set) an open subset of $\Omega$. Given $\tau>0$, let $\omega_\tau=\omega\cap((0,1)\times(0,\tau))$. We say that $\omega$ satisfies the \emph{geometric control condition (GCC) on $(0,\tau)$} if every characteristic ray intersects $\omega_\tau$. It follows from \cite[Theorem~1.8]{RLTT16} that $(B^*,A_0)$ is exactly observable on $(0,\tau)$ provided $\omega$ satisfies the GCC on $(0,\tau)$. Now suppose that $\omega$ is $\tau$-translation-invariant in the sense that $\omega+\{(0,\tau)\}=\omega$. It follows from Kronecker's theorem and a simple compactness argument that if $\tau$ is irrational then  $\omega$ satisfies the GCC on $(0,n\tau)$ for some $n\in\NN$. Hence by Proposition~\ref{prp:exp} our system is necessarily uniformly exponentially stable for such $\tau$. Since our main interest here is in non-uniform rates of convergence, we restrict our attention to the case where $\tau\in\QQ$. In fact, replacing $\tau$ by $n\tau$ for suitable $n\in\NN$ we may  further assume that we are in the resonant case where $T_0(\tau)=I$. 
We therefore assume henceforth, without essential loss of generality, that $\tau=2$. 
It then follows from Proposition~\ref{prp:spectrum} that the associated monodromy operator $U(2,0)$ is a Ritt operator.
Letting $\Omega_0=\{(s,t)\in\Omega:1<t<2\}$, we obtain the following version of Theorem~\ref{thm:gen}.

\begin{thm}\label{thm:wave}
Consider the system \eqref{eq:naCP} corresponding to the damped wave equation. Suppose that that $b$ is $2$-periodic in $t$ and let 
$$Y=\left\{x\in X:\iint_{\Omega_0}b(s,t)|v(s,t;x)|^2\,\dd (s,t)=0\right\}$$
of $X$, where $v(\cdot,\cdot;x)$ is the velocity component of the solution to the undamped wave equation on $(0,1)$ with initial data $x\in X$. Let $Z=Y^\perp$ and let $P$ denote the  orthogonal projection onto $Y$. Then for any initial value $x\in X$ the solution $z\colon\RR_+\to X$ of \eqref{eq:naCP} satisfies
\begin{equation}\label{eq:asymp_wave}
\|z(t)-z_0(t)\|\to0,\quad t\to\infty,
\end{equation}
where $z_0\colon\RR_+\to X$ is the  solution of the undamped wave equation with initial condition $z_0(0)=Px$. In particular, the system is asymptotically periodic, and it is stable if and only if  $Y=\{0\}$.

Moreover, 
\begin{equation}\label{eq:exp_wave}
\|z(t)-z_0(t)\|\le Me^{-\beta t}\|x\|,\quad t\ge0,\, x\in X,
\end{equation}
for some  $M, \beta>0$ if and only if  
\begin{equation}\label{eq:obs_wave}
\iint_{\Omega_0}b(s,t)|v(s,t;x)|^2\,\dd (s,t)\ge\kappa^2\|x\|^2
\end{equation}
for all $x\in Z$ and some $\kappa>0$. If $b=\mathbbm{1}_\omega$ for some open $2$-translation-invariant subset $\omega$ of $\Omega$, then the estimate in \eqref{eq:obs_wave} is satisfied for all $x\in X$ provided $\omega$ satisfies the GCC on $(0,2)$.  In any case,
there exists a dense subspace $X_0$ of $X$ such that for all $x\in X_0$ the convergence in \eqref{eq:asymp_wave} is superpolynomially fast.
\end{thm}

We conclude with two simple examples illustrating the way Theorem~\ref{thm:wave} can be applied in the case where $b=\mathbbm{1}_\omega$ for a $2$-translation invariant subset $\omega$ of $\Omega$. We introduce a novel approach to analysing exponential convergence to periodic orbits by studying uniform exponential stability of a related problem with a `collapsed' damping region. The collapsing technique in particular allows us to focus our attention on the complement $Z$ of the initial values resulting in non-trivial periodic orbits, and to deduce exact $Z$-observability of the original wave equation by verifying the GCC for the modified problem.

\begin{figure}[ht]
  \begin{minipage}[b]{0.32\linewidth}
    \begin{flushleft}
      \includegraphics[height=6.5cm]{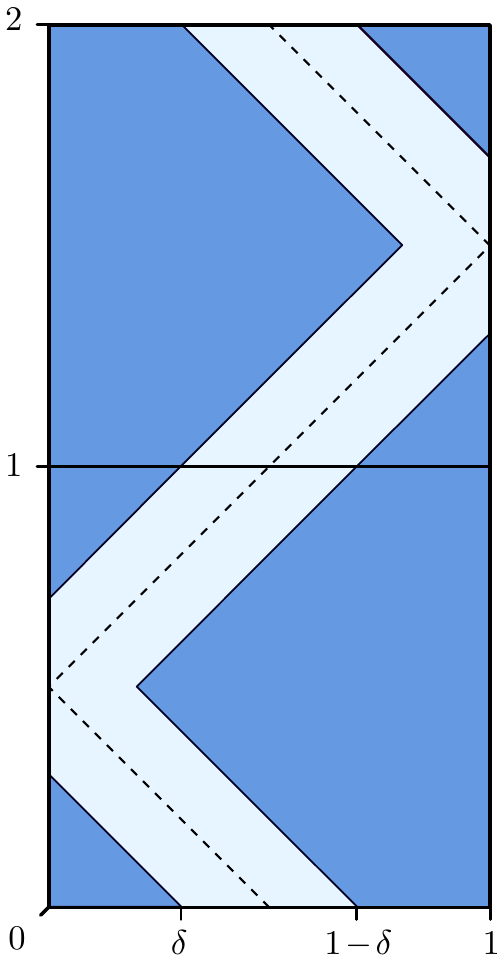}
    \end{flushleft}
  \end{minipage}
  \hfill\hfill
  \begin{minipage}[b]{0.32\linewidth}
    \begin{center}
      \includegraphics[height=6.5cm]{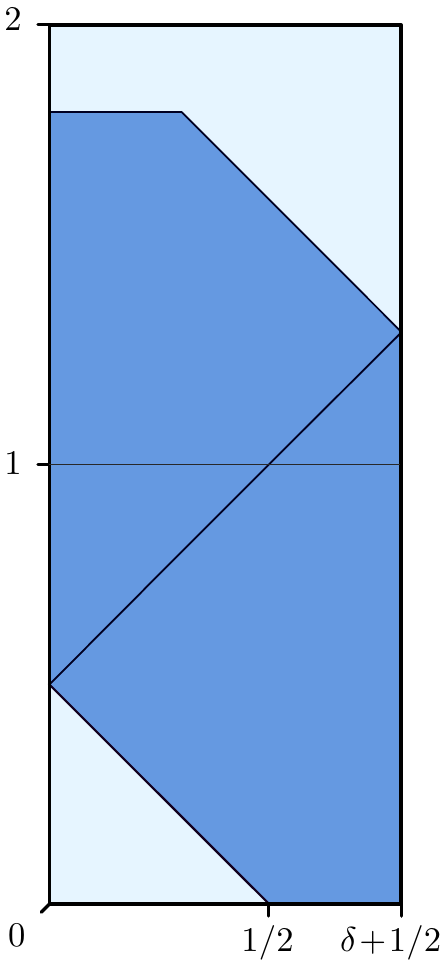}
    \end{center}
  \end{minipage}
  \begin{minipage}[b]{0.32\linewidth}
    \begin{flushright}
      \includegraphics[height=6.5cm]{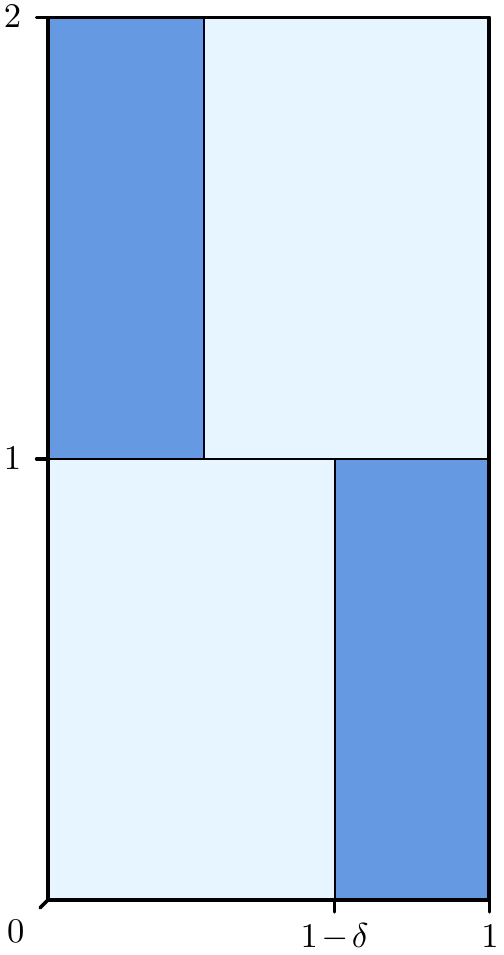}
    \end{flushright}
  \end{minipage}

    \caption{The damping regions $\omega_0$ for Examples~\ref{ex:wave1} and \ref{ex:wave2}, with $s$ on the horizontal axis and $t$ on the vertical axis. The middle picture shows the `collapsed' region considered in Example~\ref{ex:wave1}.}
      \label{fig:WavePath}
\end{figure}

\begin{ex}\label{ex:wave1}
Let $\delta\in [0,1]$ and let $p\colon[0,2]\to[0,1]$ be (part of) the characteristic ray passing through the points $(1/2,0)$, $(0,1/2)$, $(1/2,1)$, $(1,3/2)$ and $(1/2,2)$. Suppose that $\omega_0=\omega\cap\Omega_0$ is given by
$$\omega_0=\left\{(s,t)\in\Omega_0: |s-p(t)|>\frac12-\delta\right\}.$$
If $\delta\ge1/2$ then up to a null set $\omega_0$ equals $\Omega_0$. In particular, $Y=\{0\}$ and $\omega$ satisfies the GCC on $(0,2)$, so we have stability and uniform exponential convergence. Suppose now that $\delta\in[0,1/2)$ and let $I_\delta=(\delta,1-\delta)$. Then we have the orthogonal decomposition
$$Y=\left\langle \begin{pmatrix}
u_\delta\\v_\delta
\end{pmatrix}\right\rangle\oplus\left\{\begin{pmatrix}
w\\w'
\end{pmatrix}\in X: w\in H_0^1(I_\delta)\right\},$$
where 
$$u_\delta(s)=
\left\{\begin{aligned}
&s, \quad& 0<s<\delta,\\
&\textstyle{\frac{\delta(1-2s)}{1-2\delta}},& \delta<s<1-\delta,\\
&s-1,& 1-\delta<s<1,
\end{aligned}\right.
\qquad v_\delta(s)=
\left\{\begin{aligned}
&0, \quad& 0<s<\delta,\\
&\textstyle{-\frac{1}{1-2\delta}},& \delta<s<1-\delta,\\
&0,& 1-\delta<s<1.-
\end{aligned}\right.$$
In particular, the system is asymptotically periodic but not stable. In this example the orthogonal projection $P$ onto $Y$ can be computed explicitly. Indeed, for $\delta\le s\le1-\delta$ let $\phi_{\delta,s}\colon X\to\CC$ be the functional given by 
$$\phi_{\delta,s}(x)=\frac{u(s)-u(\delta)}{2}+\frac12\int_{\delta}^sv(r)\,\dd r, \quad x=\begin{pmatrix}
u\\v
\end{pmatrix}\in X,$$
and let $\psi_\delta=\phi_{\delta,1-\delta}$. Then
$$Px=-\frac{2\psi_\delta(x)}{1+2\delta}\begin{pmatrix}
u_\delta\\v_\delta
\end{pmatrix}+\begin{pmatrix}
w\\w'
\end{pmatrix},\quad x=\begin{pmatrix}
u\\v
\end{pmatrix}\in X,$$
where
$$w(s)=\phi_{\delta,s}(x)-\frac{s-\delta}{1-2\delta}\psi_\delta(x), \quad s\in I_\delta.$$
Note that the orthogonal complement $Z$ of $Y$  is given by 
$$Z=\left\{\begin{pmatrix}
u\\v
\end{pmatrix}\in X:u'+v=0 \mbox{ on } I_\delta\right\}.$$
We now show that the convergence to the periodic solution is exponentially fast, and this is achieved by `collapsing' the phase plane in such a way that the resulting damping region satisfies the GCC for the wave equation on a shorter interval. Indeed, let $J_\delta=(0,1/2+\delta)$ and $\Omega_0'=J_\delta\times(0,2)$. Moreover, let 
$$\omega_0'=\big\{(s,t)\in\Omega_0': 1/2<s+t<3/2+2\delta,\, t<3/2+\delta\big\}.$$
For $x\in Z$ it follows from a calculation based on d'Alembert's formula \eqref{eq:d'Alem} that there exists $y\in H_0^1(J_\delta)\times L^2(J_\delta)$ such that $\|y\|=\|x\|$ and
$$\iint_{\Omega_0}\mathbbm{1}_{\omega_0}|v(s,t;x)|^2\,\dd (s,t)=\iint_{\Omega_0'}\mathbbm{1}_{\omega_0'}|{v}(s,t;y)|^2\,\dd (s,t),$$
where $v(\cdot,\cdot;x)$ and $v(\cdot,\cdot;y)$ denote the velocity components of the undamped wave equation on $(0,1)$ with initial data $x$ and on $J_\delta$ with initial data $y$, respectively. Since $\omega_0'$ satisfies the GCC on $(0,2)$ for the wave equation on $J_\delta$, it follows that \eqref{eq:obs_wave} holds for some $\kappa>0$ and all $x\in Z$. In particular, we have uniform exponential convergence to the periodic solution.
\end{ex}

\begin{ex}
  \label{ex:wave2}
Let $\delta\in[0,1]$ and suppose 
that $\omega_0=\omega\cap\Omega_0$ is given by
$$\omega_0=\big((1-\delta,1)\times(0,1)\big)\cup\big((0,\delta)\times(1,2)\big).$$
This can be viewed as a model of a wave equation with \emph{switched damping}.
Note that if $\delta>0$ then separately the damping in each of the two time intervals  would lead to uniform exponential decay for all solutions. 
However, 
the periodically switched system is stable 
 if and only if $\delta\ge1/2$,
since $Y=\{0\}$ for precisely these values of $\delta$. 

If $\delta>1/2$ then (the interior of) $\omega$ satisfies the GCC on $(0,2)$, so \eqref{eq:exp_wave} holds for some $M,\beta>0$. 
For $\delta=1/2$ it is easy to see, by considering initial values of the form $x=(u,u')^T$ for $u\in H_0^1(0,1)$ with support concentrated near the point $s=1/2$, that \eqref{eq:obs_wave} does not hold, and hence nor does \eqref{eq:exp_wave}. 
Now suppose that $\delta\in(0,1/2)$, noting that $\delta=0$ corresponds to the uninteresting case of the undamped wave equation.
Letting $I_\delta=(\delta,1-\delta)$, the spaces $Y$, $Z$ and the projection $P$ are the same as in Example~\ref{ex:wave1}. By considering initial data $x\in Z$ of the form $x=(u,u')^T$ with $u\in H_0^1((0,1)\setminus I_\delta)$ having support concentrated near the points $s=\delta$ and $s=1-\delta$, it is easy to see as before that \eqref{eq:obs_wave} again fails to hold, and hence  \eqref{eq:exp_wave} does not hold either. By Remark~\ref{rem:gen}\eqref{it:slow} the convergence to the periodic solution is in fact arbitrarily slow in this case. On the other hand if $x\in X$ is of the form $x=y+z$, where $y\in Y$ and $z=(u,v)^T\in Z$ is such that $u'+v=0$ on an open interval strictly containing $I_\delta$, then by a similar `collapsing' argument to the one in Example~\ref{ex:wave1} we in fact have exponentially fast convergence to the periodic solution. 
The fact that the monodromy operator is not known explicitly in this case makes it difficult to give a precise description of those initial values $x\in X$ which lead to, say, polynomial rates of convergence to the periodic solution.
\end{ex}


\begin{thebibliography}{10}

\bibitem{AreBat88}
W.~Arendt and C.~J.~K. Batty.
\newblock Tauberian theorems and stability of one-parameter semigroups.
\newblock {\em Trans. Amer. Math. Soc.}, 306:837--841, 1988.

\bibitem{ABHN11}
W.~Arendt, C.J.K. Batty, M.~Hieber, and F.~Neubrander.
\newblock {\em Vector-valued Laplace transforms and Cauchy problems}.
\newblock Birkh\"auser, Basel, second edition, 2011.

\bibitem{BarLeb92}
C.~Bardos, G.~Lebeau, and J.~Rauch.
\newblock Sharp sufficient conditions for the observation, control, and
  stabilization of waves from the boundary.
\newblock {\em SIAM J. Control Optim.}, 30(5):1024--1065, 1992.

\bibitem{BatChi02a}
C.J.K. Batty, R.~Chill, and Y.~Tomilov.
\newblock Strong stability of bounded evolution families and semigroups.
\newblock {\em J. Funct. Anal.}, 193(1):116--139, 2002.

\bibitem{BD08}
C.J.K. Batty and T.~Duyckaerts.
\newblock Non-uniform stability for bounded semi-groups in {B}anach spaces.
\newblock {\em J. Evol. Equ.}, 8:765--780, 2008.

\bibitem{BatHut99}
C.J.K. Batty, W.~Hutter, and F.~R\"abiger.
\newblock Almost periodicity of mild solutions of inhomogeneous periodic
  {C}auchy problems.
\newblock {\em J. Differential Equations}, 156(2):309--327, 1999.

\bibitem{BenDaP07book}
A.~Bensoussan, G.~Da Prato, M.C. Delfour, and S.K. Mitter.
\newblock {\em Representation and Control of Infinite Dimensional Systems}.
\newblock Birkh{\"a}user, Boston, second edition, 2007.

\bibitem{Bur98}
N.~Burq.
\newblock D{\'e}croissance de l'{\'e}nergie locale de l'{\'e}quation des ondes
  pour le probl{\`e}me ext{\'e}rieur et absence de r{\'e}sonance au voisinage
  du r{\'e}el.
\newblock {\em Acta Mathematica}, 180(1):1--29, 1998.

\bibitem{BurGer97}
N.~Burq and P.~G\'erard.
\newblock Condition n\'ecessaire et suffisante pour la contr\^olabilit\'e
  exacte des ondes.
\newblock {\em C. R. Acad. Sci. Paris S\'er. I Math.}, 325(7):749--752, 1997.

\bibitem{CasCin14}
C.~Castro, N.~C{\^\i}ndea, and A.~M{\"u}nch.
\newblock Controllability of the linear one-dimensional wave equation with
  inner moving forces.
\newblock {\em SIAM J. Control Optim.}, 52(6):4027--4056, 2014.

\bibitem{CheFul91}
G.~Chen, S.~A. Fulling, F.~J. Narcowich, and S.~Sun.
\newblock Exponential decay of energy of evolution equations with locally
  distributed damping.
\newblock {\em SIAM Journal on Applied Mathematics}, 51(1):266--301, 1991.

\bibitem{CoLi16}
G.~Cohen and M.~Lin.
\newblock Remarks on rates of convergence of powers of contractions.
\newblock {\em J. Math. Anal. Appl.}, 436(2):1196--1213, 2016.

\bibitem{Da78}
C.M. Dafermos.
\newblock Asymptotic behavior of solutions of evolution equations.
\newblock In {\em Nonlinear evolution equations ({P}roc. {S}ympos., {U}niv.
  {W}isconsin, {M}adison, {W}is., 1977)}, volume~40 of {\em Publ. Math. Res.
  Center Univ. Wisconsin}, pages 103--123. Academic Press, New York-London,
  1978.

\bibitem{EngNag00book}
K.-J. Engel and R.~Nagel.
\newblock {\em One-Parameter Semigroups for Linear Evolution Equations}.
\newblock 2000.

\bibitem{Est84}
J.~Esterle.
\newblock Mittag-{L}effler methods in the theory of {B}anach algebras and a new
  approach to {M}ichael's problem.
\newblock In {\em Proceedings of the conference on {B}anach algebras and
  several complex variables ({N}ew {H}aven, {C}onn., 1983)}, volume~32 of {\em
  Contemp. Math.}, pages 107--129. Amer. Math. Soc., Providence, RI, 1984.

\bibitem{HaTo10}
M.~Haase and Y.~Tomilov.
\newblock Domain characterizations of certain functions of power-bounded
  operators.
\newblock {\em Studia Math.}, 196(3):265--288, 2010.

\bibitem{Ha83}
A.~Haraux.
\newblock Asymptotic behavior of trajectories for some nonautonomous, almost
  periodic processes.
\newblock {\em J. Differential Equations}, 49(3):473--483, 1983.

\bibitem{KT86}
Y.~Katznelson and L.~Tzafriri.
\newblock On power bounded operators.
\newblock {\em J. Funct. Anal.}, 68:313--328, 1986.

\bibitem{Kre85}
U.~Krengel.
\newblock {\em Ergodic Theorems}.
\newblock Walter de Gruyter, Berlin, 1985.

\bibitem{LatRan98}
Y.~Latushkin, T.~Randolph, and R.~Schnaubelt.
\newblock Exponential dichotomy and mild solutions of nonautonomous equations
  in {B}anach spaces.
\newblock {\em Journal of Dynamics and Differential Equations}, 10(3):489--510,
  1998.

\bibitem{Leb96}
G.~Lebeau.
\newblock \'{E}quation des ondes amorties.
\newblock In {\em Algebraic and geometric methods in mathematical physics
  ({K}aciveli, 1993)}, volume~19 of {\em Math. Phys. Stud.}, pages 73--109.
  Kluwer Acad. Publ., Dordrecht, 1996.

\bibitem{Lyu99}
Y.~Lyubich.
\newblock {Spectral localization, power boundedness and invariant subspaces
  under Ritt's type condition}.
\newblock {\em Studia Math.}, 134(2):153--167, 1999.

\bibitem{Mue88}
V.~M\"uller.
\newblock {Local spectral radius formula for operators in Banach spaces}.
\newblock {\em Czechoslovak Math. J.}, 38(4):726--729, 1988.

\bibitem{NaZe99}
B.~Nagy and J.~Zem\'anek.
\newblock A resolvent condition implying power boundedness.
\newblock {\em Studia Math.}, 134(2):143--151, 1999.

\bibitem{Paz83}
A.~Pazy.
\newblock {\em Semigroups of Linear Operators and Applications to Partial
  Differential Equations}.
\newblock Springer, New York, 1983.

\bibitem{RauTay74}
J.~Rauch and M.~Taylor.
\newblock Exponential decay of solutions to hyperbolic equations in bounded
  domains.
\newblock {\em Indiana Univ. Math. J.}, 24:79--86, 1974.

\bibitem{RLTT16}
J.~Le Rousseau, G.~Lebeau, P.~Terpolilli, and E.~Tr\'elat.
\newblock Geometric control condition for the wave equation with a
  time-dependent observation domain.
  \newblock {\em Anal.\ PDE}, 10(4):983--1015, 2017.

\bibitem{Sch02}
R.~Schnaubelt.
\newblock Feedbacks for nonautonomous regular linear systems.
\newblock {\em SIAM J. Control Optim.}, 41(4):1141--1165, 2002.

\bibitem{Sei16}
D.~Seifert.
\newblock {Rates of decay in the classical Katznelson-Tzafriri theorem}.
\newblock {\em J. Anal. Math.}, 130(1):329--354, 2016.

\bibitem{vN96}
J.M.A.M. van Neerven.
\newblock {\em The Asymptotic Behaviour of Semigroups of Linear Operators}.
\newblock Birkh\"auser, Basel, 1996.

\bibitem{Vu95}
Q.P. V{\~u}.
\newblock Stability and almost periodicity of trajectories of periodic
  processes.
\newblock {\em J. Differential Equations}, 115(2):402--415, 1995.

\end{thebibliography}

\end{document}